\documentclass[11pt]{article}
\usepackage{amsmath, amssymb, amsthm}
\usepackage{bbm}
\usepackage{enumitem}
\usepackage{hyperref}
\usepackage{tikz}
\usetikzlibrary{calc}
\usepackage{xcolor}
\usepackage{multirow}
\usepackage{comment}

\newtheorem{theorem}{Theorem}[section]
\newtheorem{lemma}{Lemma}[section]

\newtheorem{proposition}{Proposition}[section]
\theoremstyle{remark}
\newtheorem{remark}{Remark}

\newcommand{\N}{\mathbb{N}}

\newcommand{\Z}{\mathbb{Z}}
\newcommand{\E}{\mathbb{E}}
\newcommand{\Prob}{\mathbb{P}}

\newcommand{\leftb}[1]{\mathtt{l}_{#1}}
\newcommand{\rightb}[1]{\mathtt{r}_{#1}}

\usepackage[left=1in, right=1in, bottom = 1in, top = 1in]{geometry}
\title{Limit shape of single-source stochastic sandpiles with $p$-topplings on $\Z$}
\author{David Beck-Tiefenbach, Robin Kaiser, Julia Überbacher}

\begin{document}

\maketitle
\begin{abstract}
    We investigate the limit shape of the single-source model for stochastic sandpiles on the integer line subject to $p$--topplings. In this model, an initial configuration of $n\in\N$ particles is placed at the origin and stabilized according to a random toppling rule depending on $p\in (0,1)$: an unstable vertex sends exactly one particle to its left neighbor with probability $p$, and independently sends exactly one particle to its right neighbor with probability $p$. We prove that as $n \to \infty$, the macroscopic limit shape of the final stable configuration is a symmetric interval around the origin. Furthermore, by analyzing the center of mass martingale, we establish a central limit theorem for the boundary fluctuations, showing that after proper rescaling, they converge to a Gaussian distribution.
\end{abstract}



\section{Introduction}
Stochastic sandpiles are a randomized version of the Abelian sandpile model, which was first introduced by Bak, Tang and Wiesenfeld \cite{BTW87,BTW88} and later generalized to arbitrary graphs by Dhar \cite{D90}. In the Abelian sandpile model, vertices are declared unstable if too many particles accumulate at a single vertex, upon which the vertices are allowed to topple by sending particles to their neighbors. In this way, the model eventually reaches a critical equilibrium state; a phenomenon known as \emph{self-organized criticality}.

Whereas for the Abelian sandpile model, topplings of unstable vertices are deterministic operations, for stochastic sandpiles we allow the topplings themselves to be random. Several random toppling rules have been investigated before \cite{M91,RS12,HHRR22,CF25}, but in this paper we concentrate on $p$-topplings on the integer line $\Z$, which were first introduced and investigated in \cite{ptoppling-CMS}. On $\mathbb{Z}$, a vertex $v$ is declared unstable and legally allowed to topple if it contains at least two particles. During a $p$-toppling, $v$ acts on its adjacent neighbors by sending a particle to each neighbor independently with probability $p$, i.e. two particles are sent, one to the left and to the right each, with probability $p^2$. With probability $p(1-p)$ exactly one particle is sent to either the left or the right neighbor, while with probability $(1-p)^2$ the configuration does not change. 

Typical questions concerning stochastic sandpiles range from the behaviour of the model on finite graphs, where in each time step we add particles at random and stabilize the resulting configuration \cite{S24,SZ25,CF25} (known as \textit{driven-dissipative systems}), to the characterization of stabilization of randomly chosen configurations on infinite graphs \cite{RS12,HHRR22,CFT24} (known as \textit{fixed-energy systems}). See also \cite{LS24} for the universality conjecture, which connects these two concepts.

In the current note, we study the \textit{single-source model} on $\Z$, analyzing the final stable configuration produced by placing $n \in \N$ particles at the origin and stabilizing the system. The single-source model has previously been investigated for different particle systems, such as \textit{internal diffusion limited aggregation} (IDLA) \cite{LBG92}, \textit{rotor-router aggregation} and \textit{divisible sandpiles} \cite{LP09}. The limit shape of the single-source model for Abelian sandpiles is still an open question, the existence of the limit, in the sense that the rescaled sandpiles converge, has been proven \cite{PS13}.

Let us denote by $\eta^{(n)}:\Z\rightarrow\{0,1\}$ the final stable configuration after placing exactly $n$ particles at the origin and stabilizing. Define the left and right boundary vertices as
\begin{align}
\label{eq:boundary}\tag{BDRY}
    \leftb{n}:=\inf\{v\in\Z\;|\; \eta^{(n)}(v)>0\},&&\rightb{n}:=\sup\{v\in\Z\;|\; \eta^{(n)}(v)>0\}.
\end{align}
The first result gives the limit shape obtained after rescaling with $\frac{1}{n}$, that is, we derive a law of large numbers for the left and right boundaries.
\begin{theorem}[Limit shape of the single-source model]\label{thm:main1}
    Denote the final limit shape after stabilizing $n\in\N$ particles placed at the origin by
    $\mathcal{S}_n:=[\leftb{n},\rightb{n}]\cap\Z.$
    Then for every $\varepsilon>0$
    $$\lim_{n\rightarrow\infty}\mathbb{P}\Big([-\tfrac{n}{2}(1-\varepsilon),\tfrac{n}{2}(1-\varepsilon)]\cap\Z\subseteq \mathcal{S}_n\subseteq [-\tfrac{n}{2}(1+\varepsilon),\tfrac{n}{2}(1+\varepsilon)]\cap\Z\Big)=1.$$
\end{theorem}
That is, the final limit shape $\mathcal{S}_n$ rescaled by $\frac{1}{n}$ converges to a symmetric interval $[-\frac{1}{2},\frac{1}{2}]$ in probability. Although Theorem \ref{thm:main1} ensures sublinear boundary fluctuations for $\mathcal{S}_n$, it does not identify their exact order. As we show next, these fluctuations actually scale as $\sqrt{n}$ and satisfy a central limit theorem.
\begin{theorem}[Gaussian fluctuations at the boundary]\label{thm:main2}
    For the boundaries $\leftb{n}$ and $\rightb{n}$ of the final stable configuration after placing $n\in\N$ particles at the origin as defined in $\eqref{eq:boundary}$, it holds that
    \begin{align*}
        \frac{\leftb{n}+\frac{n}{2}}{\sqrt{n}}\xrightarrow{\mathcal{D}}\mathcal{N}\bigg(0,\frac{1-p}{12}\bigg)&&\text{and}&&\frac{\rightb{n}-\frac{n}{2}}{\sqrt{n}}\xrightarrow{\mathcal{D}}\mathcal{N}\bigg(0,\frac{1-p}{12}\bigg),
    \end{align*}
    where $\xrightarrow{\mathcal{D}}$ denotes convergence in distribution, as $n \rightarrow \infty$.
\end{theorem}
The proofs of Theorem \ref{thm:main1} and \ref{thm:main2} both rely on analyzing the center of mass defined as
{$M_n:=\sum_{x\in\Z}x\cdot\eta^{(n)}(x),$}
and using martingale techniques to show that $M_n$ obeys a law of large numbers and a central limit theorem. The claim then follows by seeing, that in the final stable configuration, at most one toppled vertex will be vacant with $0$ particles placed on it.

\textbf{Outline.} In Section \ref{sec:prelim} we define and introduce all necessary concepts to prove the two main theorems. In Section \ref{sec:lim-shape} we prove Theorem \ref{thm:main1} by analyzing how the center of mass changes during stabilization. Theorem \ref{thm:main1} then immediately follows from Proposition \ref{prop:lln}. Finally, in Section \ref{sec:gaussian} we prove Theorem \ref{thm:main2} by an application of the martingale central limit theorem to the center of mass of the final stable configuration. 
\section{Preliminaries}\label{sec:prelim}
In this section, we define the stochastic sandpile model and state the tools that are used in our proofs later.

\textbf{Stochastic sandpiles on $\Z$.}
A sandpile configuration is a function $\eta: \Z \rightarrow \N$, where for each vertex $v \in \Z$, $\eta(v)$ represents the number of particles on that vertex. A vertex $v$ is called stable if $\eta(v) \leq 1$; otherwise it is unstable. We call the configuration $\eta$ stable if each of its vertices is stable. Unstable vertices can be legally toppled. During a toppling, the toppled vertex sends particles to its neighbours at random. We consider $p$-topplings that are defined in the following manner: Fix the probability $p\in (0,1)$. When site $v\in\Z$ is toppled, it sends a particle to each neighbour with probability $p$, independently of each other.

\textbf{Diaconis-Fulton representation.}
In order to characterize the stabilization of an unstable configuration, we use the stack-wise representation that was introduced by Diaconis-Fulton \cite{DF91}. For each vertex $x \in \Z$, define a stack of toppling instructions $(\xi_{x,i})_{x \in \Z, i \in \N}$, which are sampled independently of each other, where $\xi_{x,i}$ takes values $\{\leftarrow, \rightarrow, \leftrightarrow, \varnothing\}$ with probability
\begin{equation*}
    \label{eq:toppling_probabilities}
    \Prob(\xi_{x,i} = \;\leftarrow) = \Prob(\xi_{x,i} =\; \rightarrow) = p(1-p), \qquad \Prob(\xi_{x,i} = \;\leftrightarrow) = p^2, \qquad \Prob(\xi_{x,i} = \varnothing) = (1-p)^2
\end{equation*}

Here, $\leftarrow$ denotes a toppling that sends exactly one particle to the left, $\rightarrow$ sends exactly one particle to the right, $\leftrightarrow$ sends a particle to the left and right and $\varnothing$ does not change the configuration.

For a given initial configuration $\eta$ with finite support, its stabilization is defined via an iterative procedure. First, set $\eta_0:=\eta$ and let $u_0$ be the constant zero-function on $\Z$, called the odometer function. If all the sites are stable, the procedure terminates. Otherwise, at each discrete time step $t \in \N$, select the leftmost unstable vertex in $\eta_t$ denoted by ${x_t:=\inf\{v\in\Z \mid \eta_t(v)>1\}}$ and perform a toppling by updating the configuration to
$\eta_{t+1} := \eta_t + \xi_{x_t, u_t(x_t)}$
and incrementing the odometer at $x_t$ by $1$ to obtain
$u_{t+1} = u_t + \delta_{x_t},$
where $\delta_{x_t}(x)$ is the function that is 1 if $x=x_t$ and zero otherwise. This process continues until a stable configuration is reached. In the stabilization as described above, we always choose the leftmost unstable vertex to be toppled at each time step. Due to the Abelian property of stochastic sandpiles, the distribution of the final stable configuration does not depend on the order in which we topple our unstable vertices. Indeed, let $x$ and $y$ be unstable vertices. Then, 
$(\eta + \xi_{x,i}) + \xi_{y,j} \sim (\eta + \xi_{y,j}) + \xi_{x,i},$ where $\sim$ denotes that the two quantities have the same distribution.

\textbf{Martingale Central Limit Theorem.} To prove the convergence of the boundary fluctuations, we rely on the martingale central limit theorem. We state here the version used in our paper.
\begin{theorem}[Martingale Central Limit Theorem (\cite{HaHe80}, Theorem 3.2, p. 58)]
    Let $(c_n)_{n \in \N}$ be a sequence with $c_n \uparrow \infty$ as $n \rightarrow \infty$, let  $(\mathcal{F}_{n,i})_{n\in\N,0\leq i\leq c_n}$ be sigma-fields and let $(S_{n,i})_{n\in\N,0\leq i\leq c_n}$ be a zero-mean, square-integrable martingale array with differences $X_{n,i} = S_{n,i} - S_{n,i-1}$ for $1 \leq i \leq c_n$, where it is assumed that $S_{n,0}=0$. Let $C^2$ be a constant. Suppose that as $n \to \infty$:
    \begin{align*}
        \text{1) $\max_{1\leq i \leq c_n} |X_{n,i}| \xrightarrow{\Prob} 0$},&&
        \text{2) $\sum_{1\leq i \leq c_n} X_{n,i}^2 \xrightarrow{\Prob} C^2$},&&
        \text{3) $\E\big[\max_{1\leq i \leq c_n} X_{n,i}^2\big]$ is bounded in $n$}.
    \end{align*}
     Then $\sum_{1\leq i\leq c_n} X_{n,i} \xrightarrow{\mathcal{D}} Z$, where $Z \sim \mathcal{N}(0,C^2)$.
\end{theorem}

\begin{remark}
    In \cite{HaHe80}, the martingale central limit theorem is given in a more general form, where the limit of the sum of squared differences is allowed to be a random variable $\eta^2.$ Since we only need the case where the limit is a constant $C^2$, we have stated the martingale central limit theorem for this special case.
\end{remark}

\section{Limit shape}\label{sec:lim-shape}
In this section, we determine the limit shape of the final stable configuration formed by placing $n \in \N$ particles at the origin. As outlined in the introduction, our proof relies on tracking how the center of mass changes during stabilization. To formalize this, let $\eta_0^{(n)} := n\delta_0$ denote the initial configuration, where $\delta_0$ represents the sandpile configuration given by a single particle placed at the origin.

We consider the sequence of intermediate configurations $(\eta_k^{(n)})_{k \in \N}$ from the stabilization as defined in Section \ref{sec:prelim}: if $\eta_k^{(n)}$ is unstable, we locate its leftmost unstable vertex, and obtain the next configuration by toppling $x_k^{(n)}$. If on the other hand $\eta_k^{(n)}$ is already stable, we simply set $\eta_{k+1}^{(n)} := \eta_k^{(n)}$.

Let $(\mathcal{F}_k^{(n)})_{k \in \N}$ denote the natural filtration generated by the sequence $(\eta_k^{(n)})_{k \in \N}$. Defining the center of mass at time $k \in \N$ as
\begin{align}\label{eq:center-of-mass}\tag{CM} M_k^{(n)} := \sum_{x \in \Z} x \cdot \eta_k^{(n)}(x), 
\end{align}
it follows that $\mathbb{E}[M_{k+1}^{(n)} \mid \mathcal{F}_k^{(n)}] = M_k^{(n)}$. That is, the sequence of centers of mass is a martingale. 

To establish a law of large numbers for the center of mass, we bound the total number of topplings required to stabilize the initial configuration. Define the stabilization time of $n \delta_0$ as
\begin{align}\label{eq:Kn}\tag{ST} K_n := \inf\{k \in \N \mid \eta_k^{(n)} \text{ is stable}\}. 
\end{align}

Notice that $K_n$ is a stopping time with respect to the filtration $(\mathcal{F}_{k}^{(n)})_{k\in\N}$. Furthermore, since topplings are Abelian in the stochastic sandpile model, it follows by induction that the stabilization time is finite almost surely for all $n\in\N$. That is, $K_n$ is an almost surely finite stopping time.

\begin{lemma}\label{lem:stab-time-expectation}
    For the stabilization time $K_n$ defined in (\ref{eq:Kn}), it holds that $\mathbb{E}K_n\leq\frac{n^3}{2p}$.
\end{lemma}
\begin{proof}
    Defining the square center of mass as
    \begin{align}\label{eq:square-center}\tag{SCM}S_k^{(n)}:=\sum_{x\in\Z}x^2\cdot\eta_k^{(n)},\end{align}
    it holds that
    $$\mathbb{E}[S_{k+1}^{(n)}\;|\; \mathcal{F}_k^{(n)}]=S_k^{(n)}+2p\mathbbm{1}_{\{K_n>k\}}.$$
    If we now define the sequence $\hat{S}_k^{(n)}:=S_{k}^{(n)}-2p\sum_{l=0}^{k-1}\mathbbm{1}_{\{K_n>l\}},$ then $(\hat{S}_k^{(n)})_{k\in\N}$ is a martingale with respect to the filtration $(\mathcal{F}_k^{(n)})_{k\in\N}$. The optional stopping theorem gives $\mathbb{E}[\hat{S}_{\min\{k,K_n\}}^{(n)}]=0$ for all $k\in\N$, which implies
    $$2p\mathbb{E}\bigg[\sum_{l=0}^{\min\{k,K_n\}-1}\mathbbm{1}_{\{K_n>l\}}\bigg]\leq \mathbb{E}[S_{\min\{k,K_n\}}^{(n)}]\leq n^3,$$
    and which together with the monotone convergence theorem yields
    $$\mathbb{E}K_n=\mathbb{E}\bigg[\lim_{k\rightarrow\infty}\sum_{l=0}^{\min\{k,K_n\}-1}\mathbbm{1}_{\{K_n>l\}}\bigg]=\lim_{k\rightarrow\infty}\mathbb{E}\bigg[\sum_{l=0}^{\min\{k,K_n\}-1}\mathbbm{1}_{\{K_n>l\}}\bigg]\leq \frac{n^3}{2p}.$$
\end{proof}
Using Lemma \ref{lem:stab-time-expectation}, we can derive an upper bound on the expected square of the center of mass of the final stable configuration. The center of mass of the final stable configuration is $M_{K_n}^{(n)}.$
\begin{lemma}\label{lem:square-center-mass}
    It holds that $\mathbb{E}\Big[\left(M_{K_n}^{(n)}\right)^2\Big]\leq (1-p)n^3$.
\end{lemma}
\begin{proof}
    We have
    \begin{equation}
    \label{eq:squared_difference_CM}
    \mathbb{E}\left[\left(M_{k+1}^{(n)}-M_k^{(n)}\right)^2\;|\; \mathcal{F}_k^{(n)}\right]=2p(1-p)\mathbbm{1}_{\{K_n>k\}}
    \end{equation}
    and
    \begin{align*}
        \mathbb{E}\left[\left(M_{k+1}^{(n)}\right)^2\;|\;\mathcal{F}_k^{(n)}\right]=\left(M_k^{(n)}\right)^2+\mathbb{E}\Big[\big(&M_{k+1}^{(n)}-M_k^{(n)}\big)^2\;|\; \mathcal{F}_k^{(n)}\Big] \\ &+2\mathbb{E}\left[\left(M_{k+1}^{(n)}-M_k^{(n)}\right)M_k^{(n)}\;|\;\mathcal{F}_k^{(n)}\right].
    \end{align*}
    Since the difference $M_{k+1}^{(n)}-M_k^{(n)}$ is independent of $M_k^{(n)}$ and has zero expectation, we obtain
    $$\mathbb{E}\left[\left(M_{k+1}^{(n)}\right)^2\;|\;\mathcal{F}_k^{(n)}\right]=\left(M_k^{(n)}\right)^2+2p(1-p)\mathbbm{1}_{\{K_n>k\}}.$$
    Set $\hat{M}_k^{(n)}:=\left(M_k^ {(n)}\right)^2-2p(1-p)\sum_{l=0}^{k-1}\mathbbm{1}_{\{K_n>l\}}$. Then $(\hat{M}_k^{(n)})_{k\in\N}$ is a martingale with respect to the filtration $(\mathcal{F}_k^{(n)})_{k\in\N}$. Similarly to the proof of Lemma \ref{lem:stab-time-expectation}, by the optional stopping theorem we obtain $$\mathbb{E}\left[\left(M_{\min\{k,K_n\}}^{(n)}\right)^2\right]=2p(1-p)\mathbb{E}\bigg[\sum_{l=0}^{\min\{k,K_n\}-1}\mathbbm{1}_{\{K_n>l\}}\bigg],$$
    for all $k\in\N$. The claim now follows from Fatou's Lemma together with Lemma \ref{lem:stab-time-expectation}.
\end{proof}
Recall $\leftb{n}$ and $\rightb{n}$ as defined in \eqref{eq:boundary}. We establish next a result about the placement of holes within the final stable configuration.
\begin{lemma}\label{lem:holes}
    For the final stable configuration $\eta_{K_n}^{(n)}$ it holds that either
    $$\eta_{K_n}^{(n)}=\mathbbm{1}_{[\leftb{n},\rightb{n}]\cap\Z},$$
    or there exists a vertex $\mathsf{h}_n\in[\leftb{n},\rightb{n}]\cap\Z$ such that
    $$\eta_{K_n}^{(n)}=\mathbbm{1}_{[\leftb{n},\rightb{n}]\cap\Z}-\delta_{\mathsf{h}_n}.$$
    That is, the final configuration has at most one vertex with $0$ particles within its limit shape.
\end{lemma}
\begin{proof}
    We first show that during the stabilization process, there can never be two adjacent holes. Assume that for some $k\in\N$ and $v\in\Z$ it holds that $\eta_k^{(n)}(v)=\eta_k^{(n)}(v+1)=0$. Without loss of generality, we assume that $v$ was the last of the two vertices to topple before time $k$ (the argument for $v+1$ works analogously). Since it has $0$ particles, it must have had $2$ particles before its last toppling, and it sent particles to both its neighbours. However, in that case $v+1$ received a particle, and cannot have $0$ particles at time $k$, which is a contradiction.

    It thus holds, that during the stabilization, we can never have two adjacent holes. Assume that the final stable configuration has two holes $v,w\in\Z$ with $v<w-1$, and let $\tau_w$ be the last time $w$ toppled. Since we always topple the leftmost unstable vertex, at time $\tau_w$ everything to the left of $w$ was stable. Furthermore, since $w$ is a hole, the toppling at time $\tau_w$ send a particle to the left. Since the particle that was sent to the left did not induce a sequence of topplings that filled the hole at $w$ nor at $v$, there must have been another hole $v'$ between $v$ and $w$ that absorbed the additional particle. However, by now considering the last time $t_{v'}$ the vertex $v'$ toppled before time $\tau_w$, we can repeat the argument and find another hole $v''$ between $v$ and $v'$. By repeating this argument, we must eventually find two adjacent holes during our stabilization process, which is not possible.
\end{proof}
\begin{proposition}\label{prop:lln}
    For $\leftb{n}=\inf\{v\in\Z\;|\; \eta_{K_n}^{(n)}(v)>0\}$ and $\rightb{n}=\sup\{v\in\Z\;|\;\eta_{K_n}^{(n)}(v)>0\}$, it holds that
    \begin{align*}
        \frac{\leftb{n}}{n}\xrightarrow[n\rightarrow\infty]{\mathbb{P}}-\frac{1}{2}&&\text{and}&&\frac{\rightb{n}}{n}\xrightarrow[n\rightarrow\infty]{\mathbb{P}}\frac{1}{2}.
    \end{align*}
\end{proposition}
\begin{proof}
    By Lemma \ref{lem:holes}, the final stable configuration $\eta_{K_n}^{(n)}$ is constantly $1$ on the interval of vertices from $\leftb{n}$ to $\rightb{n}$, with at most a single vertex in this interval being vacant; denote the position of this vacant vertex with $0$ particles by $\mathsf{h}_n$ (if there is no such vertex, set $\mathsf{h}_n=0$). Then
    $$M_{K_n}^{(n)}=-\mathsf{h}_n+\sum_{i=\leftb{n}}^{\rightb{n}}i=-\mathsf{h}_n+\frac{(\leftb{n}+\rightb{n})(\rightb{n}-\leftb{n}+1)}{2},$$
    and rearranging gives
    $$\frac{\leftb{n}+\rightb{n}}{n}=\frac{2M_{K_n}^{(n)}}{n(\rightb{n}-\leftb{n}+1)}+\frac{2\mathsf{h}_n}{n(\rightb{n}-\leftb{n}+1)}.$$
    It holds that $|\mathsf{h}_n|\leq n$ and $(\rightb{n}-\leftb{n}+1)\geq n$, which implies that
    $\frac{2\mathsf{h}_n}{n(\rightb{n}-\leftb{n}+1)}$ goes to $0$ almost surely.
    For the remaining summand it follows for every $\varepsilon>0$ that
    $$\mathbb{P}\left(\left|\frac{2M_{K_n}^{(n)}}{n(\rightb{n}-\leftb{n}+1)}\right|>\varepsilon\right)\leq \mathbb{P}\left(|M_{K_n}^{(n)}|>\frac{\varepsilon n^2}{2}\right)\leq\frac{4\mathbb{E}\left[\left(M_{K_n}^{(n)}\right)^2\right]}{\varepsilon^2 n^4}\leq \frac{4(1-p)}{\varepsilon^2 n},$$
    where we have used Lemma \ref{lem:square-center-mass}. Thus $(\leftb{n}+\rightb{n})/n$ goes to $0$ in probability. Combining this with the fact that $(-\leftb{n}+\rightb{n})/n$ goes to $1$ almost surely, the claim follows.
\end{proof}
\begin{proof}[Proof of Theorem \ref{thm:main1}]
    For $\varepsilon>0$ it holds
    \begin{align*}
    \Big\{\Big[-\frac{n}{2}(1-\varepsilon),\frac{n}{2}(1-\varepsilon)\Big]\cap\Z & \subseteq \mathcal{S}_n\subseteq \Big[-\frac{n}{2}(1+\varepsilon),\frac{n}{2}(1+\varepsilon)\Big]\cap\Z\Big\}\\&=\Big\{\leftb{n}\in\Big[-\frac{n}{2}(1+\varepsilon),-\frac{n}{2}(1-\varepsilon)\Big]\;,\;\rightb{n}\in\Big[\frac{n}{2}(1-\varepsilon),\frac{n}{2}(1+\varepsilon)\Big]\Big\}.
    \end{align*}
    By Proposition \ref{prop:lln}, the probability of the event on the right-hand side above goes to $1$, which proves the theorem.
\end{proof}
\section{Boundary fluctuations}\label{sec:gaussian}
This section is devoted to the proof of Theorem \ref{thm:main2}. We show that the boundary fluctuations of the limit shape are of order $\sqrt{n}$ and converge in distribution to a Gaussian random variable under appropriate rescaling. The proof is based on the martingale central limit theorem applied to the center of mass martingale from (\ref{eq:center-of-mass}).
We first show that the stabilization time $K_n$ as defined in \eqref{eq:Kn} for $n$ particles at the origin obeys a weak law of large numbers.
\begin{lemma}\label{lem:weak-law-for-Kn}
    The following convergence in probability holds
    $$\frac{K_n}{n^3}\xrightarrow[n\rightarrow\infty]{\mathbb{P}}\frac{1}{24p}\hspace{0.5cm}\text{and}\hspace{0.5cm}\frac{S_{K_n}^{(n)}}{n^3}\xrightarrow[n\rightarrow\infty]{\mathbb{P}}\frac{1}{12}.$$
\end{lemma}
\begin{proof}
    We first prove the convergence of $S_{K_n}^{(n)}$ divided by $n^3$. Recall that the final stable configuration can be described by its left and right boundary $\leftb{n}$ and $\rightb{n}$, as well as at most one vertex within the interval $[\leftb{n},\rightb{n}]$ with $0$ particles, which we will denote by $\mathsf{h}_n$. We then have
    $S_{K_n}^{(n)}=\sum_{i=1}^{\rightb{n}}i^2+\sum_{i=1}^{|\leftb{n}|}i^2-\mathsf{h}_n^2.$
    Since $\sum_{i=1}^m i^2=\frac{m^3}{3}+\frac{m^2}{2}+\frac{m}{6},$ we have
    $$\frac{1}{n^3}\sum_{i=1}^{\rightb{n}}i^2=\frac{\rightb{n}^3}{3n^3}+\frac{\rightb{n}^2}{2n^3}+\frac{\rightb{n}}{6n^3}.$$
    Because $\rightb{n}\leq n+1$, the lower order terms in the previous expression go to $0$, thus
    $$\frac{1}{n^3}\sum_{i=1}^{\rightb{n}}i^2\xrightarrow[n\rightarrow\infty]{\mathbb{P}}\frac{1}{3}\left(\frac{1}{2}\right)^3=\frac{1}{24}.$$
    The same reasoning works for $\sum_{i=1}^{|\leftb{n}|}i^2$. Since $|\mathsf{h}_n|\leq n$, we obtain $\frac{S_{K_n}^{(n)}}{n^3}\xrightarrow[n\rightarrow\infty]{\mathbb{P}}\frac{1}{12}.$

    For the second convergence, consider
    $\hat{S}_k^{(n)}:=S_{k}^{(n)}-2p\sum_{l=0}^{k-1}\mathbbm{1}_{\{K_n>l\}}$, as in the proof of Lemma \ref{lem:stab-time-expectation}. Then
    $$\mathbb{E}\left[\left(\hat{S}_{k+1}^{(n)}-\hat{S}_k^{(n)}\right)^2|\;\mathcal{F}_{k}^{(n)}\right]\leq \mathbb{E}\left[\left(S_{k+1}^{(n)}-S_k^{(n)}\right)^2|\;\mathcal{F}_{k}^{(n)}\right]+2p\mathbb{E}\left[\left|S_{k+1}^{(n)}-S_k^{(n)}\right||\;\mathcal{F}_{k}^{(n)}\right]+4p^2.$$
    If at time step $k$, the configuration topples at vertex $v$, then the change from $S_k^{(n)}$ to $S_{k+1}^{(n)}$ is bounded by $2|v|+1$. Since all toppled vertices must have absolute value less than or equal to $n$, the upper bound
    $$\mathbb{E}\left[\left(\hat{S}_{k+1}^{(n)}-\hat{S}_k^{(n)}\right)^2|\;\mathcal{F}_{k}^{(n)}\right]\leq (2n+1)^2+2p(2n+1)+4p^2<Cn^2$$
    follows, where $C>0$ is some constant. Since $(\hat{S}_{k}^{(n)})_{k\in\N}$ is a martingale, so is the sequence $((\hat{S}_{k}^{(n)})^2-\sum_{i=1}^k\mathbb{E}[(\hat{S}_{i}^{(n)}-\hat{S}_{i-1}^{(n)})^2|\;\mathcal{F}_{i-1}^{(n)}])_{k\in\N}$. Thus, the optional stopping theorem implies
    $$\mathbb{E}\left[\left(\hat{S}_{K_n}^{(n)}\right)^2\right]=\mathbb{E}\left[\sum_{i=1}^{K_n}\mathbb{E}\left[\left(\hat{S}_{i}^{(n)}-\hat{S}_{i-1}^{(n)}\right)^2|\;\mathcal{F}_{i-1}^{(n)}\right]\right]\leq Cn^2\mathbb{E}[K_n]\leq \frac{C}{2p}n^5,$$
    which together with Markov's inequality gives that
    $\frac{\hat{S}_{K_n}^{(n)}}{n^3}$ converges to $0$ in probability.
    As by definition it holds that
    $\hat{S}_{K_n}^{(n)}=S_{K_n}^{(n)}-2pK_n,$
    we thus obtain
    $\frac{K_n}{n^3}\xrightarrow[n\rightarrow\infty]{\mathbb{P}}\frac{1}{24p}.$
\end{proof}
Since we want to apply the martingale central limit theorem to the center of mass martingale, we still need to show that the sum of squared differences of the center of mass martingale converges in probability after proper rescaling.
\begin{lemma}\label{lem:squared-difference-conv}
    For the center of mass martingales it holds that
    $$\frac{1}{n^3}\sum_{i=1}^{n^5}\left(M_i^{(n)}-M_{i-1}^{(n)}\right)^2\xrightarrow[n\rightarrow\infty]{\mathbb{P}}\frac{1-p}{12}.$$
\end{lemma}
\begin{proof}
    Recalling equation \eqref{eq:squared_difference_CM}, the sequence $(W_k^{(n)})_{k\in\N}$ defined as
    $$W_k^{(n)}:=\sum_{i=1}^k\left(\left(M_i^{(n)}-M_{i-1}^{(n)}\right)^2-2p(1-p)\mathbbm{1}_{\{K_n>i-1\}}\right)$$
    is a martingale. For the squared differences we have
    $$\left(W_{k+1}^{(n)}-W_k^{(n)}\right)^2=\left(\left(M_{k+1}^{(n)}-M_k^{(n)}\right)^2-2p(1-p)\mathbbm{1}_{\{K_n>k\}}\right)^2\leq4,$$
    thus for every $k\in\N$ it holds
    $\mathbb{E}\left[\left(W_{k}^{(n)}\right)^2\right]\leq 4k.$
    Markov's inequality yields that
    $W_{n^5}^{(n)}/n^3$ converges to $0$ in probability.
    Lemma \ref{lem:weak-law-for-Kn} gives that
    $$\frac{1}{n^3}\sum_{i=1}^{n^5}\mathbbm{1}_{\{K_n>i-1\}}\xrightarrow[n\rightarrow\infty]{\Prob}\frac{1}{24p}$$
    which finally implies
    $$\frac{1}{n^3}\sum_{i=1}^{n^5}\left(M_i^{(n)}-M_{i-1}^{(n)}\right)^2\xrightarrow[n\rightarrow\infty]{\mathbb{P}}2p(1-p)\frac{1}{24p}=\frac{1-p}{12}.$$
\end{proof}
We are now in a position to prove the central limit theorem for the center of mass of the final stable configuration.
\begin{proposition}
    For the sequence of centers of mass of the final stable configuration it holds that
    $$\frac{M_{K_n}^{(n)}}{n^{3/2}}\xrightarrow{\mathcal{D}}\mathcal{N}\Big(0,\frac{1-p}{12}\Big)$$
    as $n \rightarrow \infty.$
\end{proposition}
\begin{proof}
   We first show the statement for  $(M_{n^5}^{(n)})_{n\in\N}$, by checking the conditions of the martingale central limit theorem. Since for any $k\in\N$ and all $n\in\N$, the martingale difference
    $$|M_{k+1}^{(n)}-M_k^{(n)}|\leq 1$$
    is always bounded in absolute value, we easily obtain
    $$\max_{k\leq n^5}\left|\frac{M_{k+1}^{(n)}-M_k^{(n)}}{n^{3/2}}\right|\xrightarrow[n\rightarrow\infty]{\mathbb{P}}0\hspace{0.5cm}\text{and}\hspace{0.5cm}\mathbb{E}\Bigg[\max_{k\leq n^5}\frac{\left(M_{k+1}^{(n)}-M_k^{(n)}\right)^2}{n^3}\Bigg]\text{ is bounded}.$$
    Furthermore, Lemma \ref{lem:squared-difference-conv} gives
    $$\frac{1}{n^3}\sum_{i=1}^{n^5}\left(M_i^{(n)}-M_{i-1}^{n)}\right)^2\xrightarrow[n\rightarrow\infty]{\mathbb{P}}\frac{1-p}{12}.$$
    We can thus apply the martingale central limit theorem to obtain
    $$\frac{M_{n^5}^{(n)}}{n^{3/2}}\xrightarrow{\mathcal{D}}\mathcal{N}\Big(0,\frac{1-p}{12}\Big)$$
    as $n \rightarrow \infty.$ It remains to exchange $n^5$ with the stopping time $K_n$. Since
    $$\mathbb{P}\left(\frac{M_{n^5}^{(n)}}{n^{3/2}}\neq \frac{M_{K_n}^{(n)}}{n^{3/2}}\right)\leq\mathbb{P}(K_n>n^5)\leq \frac{\mathbb{E}K_n}{n^5}\leq\frac{1}{2pn^2},$$
    where above we have used Lemma \ref{lem:stab-time-expectation} in the last inequality, it follows that the same convergence to the normal distribution holds for the martingale stopped at $K_n$. This finishes the proof.
\end{proof}

\begin{proof}[Proof of Theorem \ref{thm:main2}]
    We establish the convergence to a normal distribution only for the right boundary $\rightb{n}$, the proof for the left boundary works analogously. We have
    $$\frac{\rightb{n}-n/2}{\sqrt{n}}=\frac{\rightb{n}-\leftb{n}-n}{2\sqrt{n}}+\frac{\leftb{n}+\rightb{n}}{2\sqrt{n}}.$$
    Since $n\leq\rightb{n}-\leftb{n}\leq n+1$, we obtain
    $$\lim_{n\rightarrow\infty}\frac{\rightb{n}-\leftb{n}-n}{2\sqrt{n}}=0.$$
    For the remaining term, analogously to the proof of Proposition \ref{prop:lln}, we get
    $$\frac{\leftb{n}+\rightb{n}}{2\sqrt{n}}=\frac{n}{(\rightb{n}-\leftb{n}+1)}\frac{M_{K_n}^{(n)}}{n^{3/2}}-\frac{\mathsf{h}_n}{\sqrt{n}(\rightb{n}-\leftb{n}+1)},$$
    where $\mathsf{h}_n$ denotes the position of the vacant vertex in the final stable configuration. For the second summand above we have
    $$\frac{\mathsf{h}_n}{\sqrt{n}(\rightb{n}-\leftb{n}+1)}\leq\frac{n}{n^{3/2}}\rightarrow0.$$
    Finally, since
    $$\lim_{n\rightarrow\infty}\frac{n}{(\rightb{n}-\leftb{n}+1)}=1,$$
    we obtain from Slutsky's Theorem that
    $$\frac{\rightb{n}-n/2}{\sqrt{n}}\xrightarrow{\mathcal{D}}\mathcal{N}\Big(0,\frac{1-p}{12}\Big)$$
    as $n \rightarrow \infty.$
\end{proof}
\section{Conclusion}
In this paper, we established that the limit shape of the single-source stochastic sandpile model with $p$-topplings on $\Z$ is a symmetric interval, and the fluctuations at the boundaries converge to a Gaussian distribution. A natural direction for future research is to consider the single source model with different toppling rules. Our current analysis heavily relies on the property that the final stable configuration contains at most one hole. For toppling rules where particles can travel distances greater than one, this property may fail. Consequently, analyzing such models would first require rigorous bounds on the size and distribution of holes within the limit shape.

Additionally, investigating the single-source model on other state spaces, particularly $\Z^2$, presents a compelling challenge.
The limit shape of the single-source Abelian sandpile model in two dimensions is conjectured to deviate from a Euclidean ball. It is a challenging problem to consider the limit shape of single-source stochastic sandpiles with $p$-topplings in two  dimensions. We believe that due to the randomness in the toppling rule, the two-dimensional limit shape for stochastic sandpiles will be the Euclidean sphere.
\newline
\newline
\textbf{Funding Information.} This research was funded in part by the Austrian Science Fund (FWF) 10.55776/P34713 and 10.55776/PAT3123425.

\bibliographystyle{alpha}
\bibliography{lit}

@article {BTW88,
    AUTHOR = {Bak, Per and Tang, Chao and Wiesenfeld, Kurt},
     TITLE = {Self-organized criticality},
   JOURNAL = {Phys. Rev. A (3)},
  FJOURNAL = {Physical Review. A. Third Series},
    VOLUME = {38},
      YEAR = {1988},
    NUMBER = {1},
     PAGES = {364--374},
      ISSN = {1050-2947,1094-1622},
   MRCLASS = {58F13 (82A25 92A05)},
  MRNUMBER = {949160},
       DOI = {10.1103/PhysRevA.38.364},
       URL = {https://doi.org/10.1103/PhysRevA.38.364},
}

@article{BTW87,
  title = {Self-organized criticality: An explanation of the 1/f noise},
  author = {Bak, Per and Tang, Chao and Wiesenfeld, Kurt},
  journal = {Phys. Rev. Lett.},
  volume = {59},
  issue = {4},
  pages = {381--384},
  numpages = {0},
  year = {1987},
  month = {Jul},
  publisher = {American Physical Society},
  doi = {10.1103/PhysRevLett.59.381},
  url = {https://link.aps.org/doi/10.1103/PhysRevLett.59.381}
}

@article {D90,
    AUTHOR = {Dhar, Deepak},
     TITLE = {Self-organized critical state of sandpile automaton models},
   JOURNAL = {Phys. Rev. Lett.},
  FJOURNAL = {Physical Review Letters},
    VOLUME = {64},
      YEAR = {1990},
    NUMBER = {14},
     PAGES = {1613--1616},
      ISSN = {0031-9007},
   MRCLASS = {82A68 (82A60)},
  MRNUMBER = {1044086},
       DOI = {10.1103/PhysRevLett.64.1613},
       URL = {https://doi.org/10.1103/PhysRevLett.64.1613},
}

@article{M91,
doi = {10.1088/0305-4470/24/7/009},
url = {https://doi.org/10.1088/0305-4470/24/7/009},
year = {1991},
month = {apr},
publisher = {},
volume = {24},
number = {7},
pages = {L363},
author = {S S Manna},
title = {Two-state model of self-organized criticality},
journal = {Journal of Physics A: Mathematical and General}
}

@article {RS12,
    AUTHOR = {Rolla, Leonardo T. and Sidoravicius, Vladas},
     TITLE = {Absorbing-state phase transition for driven-dissipative
              stochastic dynamics on {${\Bbb Z}$}},
   JOURNAL = {Invent. Math.},
  FJOURNAL = {Inventiones Mathematicae},
    VOLUME = {188},
      YEAR = {2012},
    NUMBER = {1},
     PAGES = {127--150},
      ISSN = {0020-9910,1432-1297},
   MRCLASS = {60K35 (82C22)},
  MRNUMBER = {2897694},
MRREVIEWER = {Andrew\ R.\ Wade},
       DOI = {10.1007/s00222-011-0344-5},
       URL = {https://doi-org.tum-eaccess.de/10.1007/s00222-011-0344-5},
}

@ARTICLE{HHRR22,
       author = {{Hoffman}, Christopher and {Hu}, Yiping and {Richey}, Jacob and {Rizzolo}, Douglas},
        title = "{Active Phase for the Stochastic Sandpile on Z}",
      journal = {arXiv e-prints},
         year = 2022,
        month = dec,
          eid = {arXiv:2212.08293},
          note={arXiv:2212.08293},
          doi = {10.48550/arXiv.2212.08293},
archivePrefix = {arXiv},
       eprint = {2212.08293},
 primaryClass = {math.PR},
       adsurl = {https://ui.adsabs.harvard.edu/abs/2022arXiv221208293H}
}

@article{CF25, 
      author={Concetta Campailla and Nicolas Forien},
      title="{Stochastic Sandpile Model: exact sampling and complete graph}",
      journal={arXiv e-prints},
      year=2026,
      eid={arXiv:2507.01572},
      note={arXiv:2507.01572},
      doi={10.48550/arXiv.2507.01572},
      archivePrefix={arXiv},
      eprint={2507.01572},
      primaryClass={math.PR},
      url={https://arxiv.org/abs/2507.01572}, 
}

@article {ptoppling-CMS,
    AUTHOR = {Chan, Yao-ban and Marckert, Jean-Fran\c cois and Selig,
              Thomas},
     TITLE = {A natural stochastic extension of the sandpile model on a
              graph},
   JOURNAL = {J. Combin. Theory Ser. A},
  FJOURNAL = {Journal of Combinatorial Theory. Series A},
    VOLUME = {120},
      YEAR = {2013},
    NUMBER = {7},
     PAGES = {1913--1928},
      ISSN = {0097-3165,1096-0899},
   MRCLASS = {05C31 (05C30)},
  MRNUMBER = {3092706},
MRREVIEWER = {Fu\ Ji\ Zhang},
       DOI = {10.1016/j.jcta.2013.07.004},
       URL = {https://doi-org.tum-eaccess.de/10.1016/j.jcta.2013.07.004},
}

@ARTICLE{CFT24,
       author = {{Campailla}, Concetta and {Forien}, Nicolas and {Taggi}, Lorenzo},
        title = "{The critical density of the Stochastic Sandpile Model}",
      journal = {arXiv e-prints},
         year = 2024,
        month = oct,
          eid = {arXiv:2410.18905},
          note={arXiv:2410.18905},
          doi = {10.48550/arXiv.2410.18905},
archivePrefix = {arXiv},
       eprint = {2410.18905},
 primaryClass = {math.PR},
       adsurl = {https://ui.adsabs.harvard.edu/abs/2024arXiv241018905C}
}

@article {S24,
    AUTHOR = {Selig, Thomas},
     TITLE = {The stochastic sandpile model on complete graphs},
   JOURNAL = {Electron. J. Combin.},
  FJOURNAL = {Electronic Journal of Combinatorics},
    VOLUME = {31},
      YEAR = {2024},
    NUMBER = {3},
     PAGES = {Paper No. 3.26, 29},
      ISSN = {1077-8926},
   MRCLASS = {05A15 (05A19 60J10)},
  MRNUMBER = {4794445},
MRREVIEWER = {Luis\ Verde-Star},
       DOI = {10.37236/12780},
       URL = {https://doi-org.tum-eaccess.de/10.37236/12780},
}

@incollection {SZ25,
    AUTHOR = {Selig, Thomas and Zhu, Haoyue},
     TITLE = {Abelian and stochastic sandpile models on complete bipartite
              graphs},
 BOOKTITLE = {W{ALCOM}: algorithms and computation},
    SERIES = {Lecture Notes in Comput. Sci.},
    VOLUME = {15411},
     PAGES = {326--345},
 PUBLISHER = {Springer, Singapore},
      YEAR = {[2025] \copyright 2025},
      ISBN = {978-981-96-2844-5; 978-981-96-2845-2},
   MRCLASS = {68R10},
  MRNUMBER = {4887055},
       DOI = {10.1007/978-981-96-2845-2\_21},
       URL = {https://doi-org.tum-eaccess.de/10.1007/978-981-96-2845-2_21},
}

@article {LS24,
    AUTHOR = {Levine, Lionel and Silvestri, Vittoria},
     TITLE = {Universality conjectures for activated random walk},
   JOURNAL = {Probab. Surv.},
  FJOURNAL = {Probability Surveys},
    VOLUME = {21},
      YEAR = {2024},
     PAGES = {1--27},
      ISSN = {1549-5787},
   MRCLASS = {60K35 (82B26 82B27 82C22 82C26)},
  MRNUMBER = {4718500},
       DOI = {10.1214/24-ps25},
       URL = {https://doi-org.tum-eaccess.de/10.1214/24-ps25},
}

@article {LBG92,
    AUTHOR = {Lawler, Gregory F. and Bramson, Maury and Griffeath, David},
     TITLE = {Internal diffusion limited aggregation},
   JOURNAL = {Ann. Probab.},
  FJOURNAL = {The Annals of Probability},
    VOLUME = {20},
      YEAR = {1992},
    NUMBER = {4},
     PAGES = {2117--2140},
      ISSN = {0091-1798,2168-894X},
   MRCLASS = {60J15 (60K35 82B41)},
  MRNUMBER = {1188055},
MRREVIEWER = {H.\ Kesten},
       URL =
              {http://links.jstor.org.tum-eaccess.de/sici?sici=0091-1798(199210)20:4<2117:IDLA>2.0.CO;2-K&origin=MSN},
}

@article {LP09,
    AUTHOR = {Levine, Lionel and Peres, Yuval},
     TITLE = {Strong spherical asymptotics for rotor-router aggregation and
              the divisible sandpile},
   JOURNAL = {Potential Anal.},
  FJOURNAL = {Potential Analysis. An International Journal Devoted to the
              Interactions between Potential Theory, Probability Theory,
              Geometry and Functional Analysis},
    VOLUME = {30},
      YEAR = {2009},
    NUMBER = {1},
     PAGES = {1--27},
      ISSN = {0926-2601,1572-929X},
   MRCLASS = {60G50},
  MRNUMBER = {2465710},
       DOI = {10.1007/s11118-008-9104-6},
       URL = {https://doi-org.tum-eaccess.de/10.1007/s11118-008-9104-6},
}

@article {PS13,
    AUTHOR = {Pegden, Wesley and Smart, Charles K.},
     TITLE = {Convergence of the {A}belian sandpile},
   JOURNAL = {Duke Math. J.},
  FJOURNAL = {Duke Mathematical Journal},
    VOLUME = {162},
      YEAR = {2013},
    NUMBER = {4},
     PAGES = {627--642},
      ISSN = {0012-7094,1547-7398},
   MRCLASS = {60K35 (35D40 35J60 35R35 60J60)},
  MRNUMBER = {3039676},
MRREVIEWER = {Antal\ A.\ J\'arai},
       DOI = {10.1215/00127094-2079677},
       URL = {https://doi-org.tum-eaccess.de/10.1215/00127094-2079677},
}

@book {HaHe80,
    AUTHOR = {Hall, P. and Heyde, C. C.},
     TITLE = {Martingale limit theory and its application},
    SERIES = {Probability and Mathematical Statistics},
 PUBLISHER = {Academic Press, Inc. [Harcourt Brace Jovanovich, Publishers],
              New York-London},
      YEAR = {1980},
     PAGES = {xii+308},
      ISBN = {0-12-319350-8},
   MRCLASS = {60-02 (60B12 60F05 60G42)},
  MRNUMBER = {624435},
MRREVIEWER = {David\ J.\ Aldous},
}

@incollection {DF91,
    AUTHOR = {Diaconis, P. and Fulton, W.},
     TITLE = {A growth model, a game, an algebra, {L}agrange inversion, and
              characteristic classes},
      NOTE = {Commutative algebra and algebraic geometry, II (Italian)
              (Turin, 1990)},
   JOURNAL = {Rend. Sem. Mat. Univ. Politec. Torino},
  FJOURNAL = {Universit\`a{} e Politecnico di Torino. Seminario Matematico.
              Rendiconti},
    VOLUME = {49},
      YEAR = {1991},
    NUMBER = {1},
     PAGES = {95--119},
      ISSN = {0373-1243},
   MRCLASS = {60J10 (05A15 14F25)},
  MRNUMBER = {1218674},
}

\textsc{David Beck-Tiefenbach}, Universität Innsbruck, Institut für Mathematik, Technikerstraße 23, A-6020 Innsbruck, Austria. \texttt{david.beck-tiefenbach@uibk.ac.at}

\textsc{Robin Kaiser}, Departement of Mathematics, CIT, Technische Universität München, Boltzmannstr. 3, D-85748 Garching bei München, Germany. \texttt{ro.kaiser@tum.de}

\textsc{Julia Überbacher}, Universität Innsbruck, Institut für Mathematik, Technikerstraße 23, A-6020 Innsbruck, Austria. \texttt{julia.ueberbacher@uibk.ac.at}
\end{document}